\documentclass{amsart}
\usepackage{mathrsfs}
\usepackage{amsfonts}
\usepackage{amssymb}
\usepackage{amsthm}
\usepackage[papersize={6.9in,9.8in},textwidth=14cm,textheight=21.7cm,centering]{geometry}

\usepackage[colorlinks=true,citecolor=blue,linkcolor=blue]{hyperref}

\usepackage{hyperref}

\allowdisplaybreaks

\newtheorem*{theorem*}{Theorem}
\newtheorem{thm}{Theorem}
\newtheorem{lem}{Lemma}

\newtheorem*{proposition*}{Proposition}
\newtheorem{prop}{Proposition}

\newtheorem*{corollary*}{Corollary}
\newtheorem{corollary}{Corollary}

\theoremstyle{remark}
\newtheorem{rem}{Remark}

\begin{document}

\title[On the mean value of a kind of Zeta functions]{On the mean value of a kind of Zeta functions}

\author{Kui Liu}

\address{Department of Mathematics, Qingdao University, Qingdao 266071, P. R. China}

\email{email: liukui@qdu.edu.cn}

%\subjclass[2010]{}

\keywords{Mean value, Zeta function, Voronoi formula, primitive Pythagorean triangles.}

\begin{abstract}\label{Definition of Zeta_ab function}
Let $d_{\alpha ,\beta }(n)=\sum\limits_{\substack{ n=kl  \\ \alpha l<k\leq
\beta l}}1$ be the number of ways of factoring n into two almost equal integers.
For rational numbers $0<\alpha <\beta $, we consider the following Zeta function
$\zeta _{\alpha ,\beta }(s)=\sum\limits_{n=1}^{\infty }\frac{d_{\alpha ,\beta }(n)}{n^{s}}$ for $\Re s>1.$
It has an analytic continuation to $\Re s>1/3.$ We get an asymptotic formula for the mean square of $\zeta _{\alpha,\beta }(s)$ in the strip $1/2<\Re s<1$. As an application, we improve an result on the distribution of primitive Pythagorean triangles.
\end{abstract}

\maketitle

\section{Introduction and main results}

All through this paper, we always suppose $s=\sigma +it$ and $x\geq 2$.  Let
\begin{equation*}
d(n)=\sum
\limits_{n=kl}1
\end{equation*}
be the classical divisor function and
\begin{equation*}
D(n)=\sum_{n\leq x}d(n)
\end{equation*}
be its summatory function.
Dirichlet proved
\begin{equation}\label{Sum of d(n)}
D(x)=x(\log x+2\gamma -1)+\Delta (x),
\end{equation}%
where $\gamma =\lim\limits_{n\rightarrow \infty }\left( \sum\limits_{k=1}^{n}\frac{1}{k}-\log n\right) \approx 0.5721\cdots $ is the
Euler constant and
\begin{equation*}
\Delta (x)\ll x^{\frac{1}{2}}.
\end{equation*}
Voronoi \cite{Voronoi} improved Dirichlet's result to
\begin{equation*}
\Delta (x)\ll x^{\frac{1}{3}}\log x.
\end{equation*}
It is conjectured that for any $\varepsilon >0$, we have
\begin{equation*}
\Delta (x)\ll_{\varepsilon} x^{
\frac{1}{4}+\varepsilon }.
\end{equation*}
The best result to date is
\begin{equation*}
\Delta (x)\ll x^{\frac{131}{416}}(\log x)^\frac{26947}{8320},
\end{equation*}%.
due to Huxley \cite{Huxley}.
Let $\zeta \left( s\right) $ be the Riemann Zeta-function, then the generated function of $d(n)$ is
\begin{equation*}
\zeta ^{2}\left( s\right) =\sum\limits_{n=1}^{\infty }\frac{d\left(n\right) }{n^{s}},{\rm\ \ \ \ for\ }\sigma>1.
\end{equation*}
 Hardy-Littlewood \cite{Hardy and Littlewood}
considered the mean square of $\zeta^{2}\left( s\right)$
\begin{equation*}
I_{\sigma }\left( T,\zeta^{2}\right) =\int_{T}^{2T}\left\vert \zeta\left( \sigma
+it\right) \right\vert ^{4}dt,\text{\ \ \ \ for\ } 1/2<\sigma<1,
\end{equation*}
and proved
\begin{equation}\label{Fourth moments for Zeta (s)}
I_{\sigma }\left( T,\zeta ^{2}\right) =\frac{\zeta ^{4}\left( 2\sigma
\right) }{\zeta \left( 4\sigma \right) }T+o\left( T\right).
\end{equation}
Note that their proof is based on the approximation (for example, see Section 3 of \cite{Ivic})
\begin{equation}\label{approximation for square of Zeta(x)}
\zeta ^{2}\left( s\right) =\sum_{n\leq x}\frac{d\left( n\right) }{n^{s}}
+\chi ^{2}\left( s\right) \sum_{n\leq y}\frac{d\left( n\right) }{n^{1-s}}
+O\left( x^{\frac{1}{2}-\sigma }\log t\right), \text{\ \ \ \ for\ } 1/2<\sigma<1,
\end{equation}
where $x,y\geq 2$, $4\pi^2 xy=t^{2}$ and
\begin{equation*}
\chi \left( s\right)=\frac{\left( 2\pi \right) ^{s}}{2\Gamma \left(
s\right) \cos \left( \frac{\pi s}{2}\right) }
\end{equation*}
is the $\Gamma $-factor in the functional equation%
\begin{equation}\label{Functional equation of Riemann zeta-function}
\zeta \left( s\right) =\chi \left( s\right) \zeta \left( 1-s\right).
\end{equation}

In this paper, we focus on the following type divisor function given by
\begin{equation*}
d_{\alpha ,\beta }(n)=\sum_{\substack{ n=kl  \\ \alpha l<k\leq
\beta l}}1,
\end{equation*}
where $\alpha ,\beta $ are fixed rational numbers satisfying
$0<\alpha <\beta $. Define its generated Zeta function as
\begin{equation*}
\zeta _{\alpha ,\beta }(s)=\sum_{n=1}^{\infty }\frac{d_{\alpha ,\beta }(n)}{
n^{s}},{\rm\ \ \ \ for\ }\sigma>1.
\end{equation*}
We prove that $\zeta _{\alpha ,\beta }(s)$ has an analytic continuation to $\sigma>1/3$ and get an asymptotic formula for the mean square of
$\zeta _{\alpha ,\beta }(s)$ in the strip $1/2<\sigma<1$.

\begin{thm}\label{Main theorem}
For any $\frac{1}{2}<\sigma <1$ and rational numbers $0<\alpha<\beta$, there exists a constant $\varepsilon \left(\sigma\right) >0$ such that
\begin{equation}
\int_{T}^{2T}\left\vert \zeta _{\alpha ,\beta }(\sigma
+it)\right\vert
^{2}dt=T\sum_{n=1}^{\infty }\frac{d_{\alpha ,\beta }^{2}(n)}{n^{2\sigma }}%
+O_{\alpha ,\beta ,\sigma }\left( T^{1-\varepsilon \left( \sigma
\right) }\right) .  \label{Main asymptotic formula}
\end{equation}
\end{thm}

Theorem \ref{Main theorem} can be used to study the distribution of primitive Pythagorean triangles (i.e. triples $%
\left( a,b,c\right) $ with $a,b,c\in\mathbb{N},$ $a^{2}+b^{2}=c^{2},$ $a<b$ and $\gcd \left( a,b,c\right) =1$). Let $P(x)$
denote the number of primitive Pythagorean triangles with perimeter $a+b+c\leq x.$
D. H. Lehmer \cite{H.Lehmer} proved%
\begin{equation*}
P\left( x\right) =\frac{\log 2}{\pi ^{2}}x+O\left( x^{1/2}\log x\right) .
\end{equation*}
It is difficult to reduce the exponents $1/2$ in the error term, which
depends on the zero-free region of the Riemann zeta function. However, assuming the
Riemann Hypothesis, it was showed in \cite{Liu} that, for any $\varepsilon >0,$ we have
\begin{equation}
\label{old theorem on triangles}
P\left( x\right) =\frac{\log 2}{\pi ^{2}}x+O_{\varepsilon }\left( x^{\frac{%
5805}{15408}+\varepsilon }\right) .
\end{equation}
We improve this result by applying Theorem \ref{Main theorem} and get

\begin{thm}
\label{Theorem 1.2}
If the Riemann Hypothesis is true, then for any $\varepsilon >0,$ we have
\begin{equation*}
P\left( x\right) =\frac{\log 2}{\pi ^{2}}x+O_{\varepsilon }\left( x^{\frac{4%
}{11}+\varepsilon }\right) .
\end{equation*}
\end{thm}
\noindent Note that $\frac{5805}{15408}=0.3767\cdots $ and $\frac{4}{11}=0.3636\cdots .$

\bigskip

\section{Main steps in the proof of Theorem \ref{Main
theorem}}

First, Let's recall a way to get the asymptotic formula (\ref{Fourth moments for Zeta (s)}). In Chapters 3 of \cite{Ivic}, using the functional equation (\ref{Functional equation of Riemann zeta-function}), Ivic derive the
Voronoi formula for the error term $\Delta \left( x\right) $ in (\ref{Sum of d(n)}). Then in Chapter 4 of \cite{Ivic},
Ivic get the approximation (\ref{approximation for square of Zeta(x)}) by the Voronoi formula,
from which one can obtain (\ref{Fourth moments for Zeta (s)}) in a standard way.

Now observing that $\zeta _{\alpha ,\beta }(s)$ is similar to $\zeta ^{2}(s),$ we can realize
$\int_{T}^{2T}\left\vert \zeta _{\alpha ,\beta }(\sigma+it)\right\vert ^{2}dt $ as an analogue of
$\int_{T}^{2T}\left\vert \zeta (\sigma +it)\right\vert ^{4}dt.$
Our main steps in the proof of Theorem \ref{Main theorem} similar to the proof of (\ref{Fourth moments for Zeta (s)}).
In Section 4, we study the asymptotic property of the summatory
function
\begin{equation}
D_{\alpha ,\beta }(x)=\sum_{n\leq x}d_{\alpha ,\beta }(n).
\label{Summatory function of d_ab(n)}
\end{equation}%
In Section 5, we derive a Voronoi type formula for the error term
\begin{equation*}
\Delta _{\alpha ,\beta }(x)=D_{\alpha ,\beta }(x)-\text{Main
terms.}
\end{equation*}%
In Section 6, using the asymptotic formula of $D_{\alpha ,\beta }(x)$ and the Voronoi type formula for $\Delta_{\alpha ,\beta }(x)$, we obtain
the following approximation for $%
\zeta _{\alpha ,\beta }(s)$, which is the key for the proof of Theorem
\ref{Main theorem}.
\begin{prop}\label{Proposition of Approximation}
The function $\zeta _{\alpha ,\beta }(s)$ can be analytically extended to the half plane $%
\sigma >\frac{1}{3}$ with simple poles at $s=\frac{1}{2},1.$ Moreover, suppose $
T\geq 2,$ $s=\sigma +it$ and $4\pi ^{2}xy=t^{2}$, then for any $\frac{1}{2}<\sigma
<1,$ $T<t\leq
2T$ and $0<\alpha<\beta$, we have
\begin{equation}\label{Approximation for Zeta_ab}
\zeta _{\alpha ,\beta }(s)=\sum_{n\leq x}\frac{d_{\alpha ,\beta }(n)}{n^{s}}%
+\chi ^{2}\left( s\right) \sum_{n\leq y}\frac{d_{\alpha ,\beta }(n)}{n^{s-1}}%
+E_{\alpha ,\beta }\left( s\right) ,
\end{equation}
where $\chi\left( s\right)$ is given by \rm{(\ref{Functional equation of Riemann zeta-function})} and $E_{\alpha ,\beta }\left( s\right) $ satisfies
\begin{equation}\label{Meansquare for E_ab}
\int_{T}^{2T}\left\vert E_{\alpha ,\beta }(\sigma +it)\right\vert
^{2}dt\ll
_{\alpha ,\beta ,\sigma }\left( x^{-2\sigma }T^{2}+x^{1-\sigma }T^{\frac{1}{2%
}}+x^{\frac{1}{2}-\sigma }T+x^{-\sigma }T^{\frac{3}{2}}\right) \log
^{3}T.
\end{equation}
\end{prop}
From (\ref{Approximation for Zeta_ab}), we can derive Theorem \ref{Main theorem} in a standard way.
Hence the main work of paper is to prove Proposition \ref{Proposition of Approximation}.

\bigskip

\section{Priliminary lemmas}

Denote the integral part of $u$ by $[u]$. let $\psi \left( u\right)=u-[u]-\frac{1}{2}$ and $e(x)=e^{2\pi ix}$.
It is well known that $\psi \left( u\right)$ has a truncated Fourier expansion
(for example, see \cite{Heath-Brown}).
\begin{lem}\label{Fourier expansion of Psi}
For any real number $H>2,$ we have
\begin{equation*}
\psi \left( u\right)=-\frac{1}{2\pi i}\sum_{1\leq \left\vert
h\right\vert \leq H}\frac{1}{h}e\left( hu\right) +O\left( G\left(
u,H\right) \right),
\end{equation*}
where
\begin{equation}\label{Def. of G(u,H)}
G\left(u,H\right) =\min \left( 1,\frac{1}{H\left\vert \left\vert
u\right\vert \right\vert }\right).
\end{equation}
\end{lem}

\noindent We will use the first derivative test (for example, see Chapter 21 of \cite{Pan and Pan}).
\begin{lem}\label{First drivitive test}
Let $G\left( x\right) $ and $F\left(x\right) $ be a real differentiable functions such that
$\frac{F^{\prime }\left( x\right) }{G\left( x\right) }$ is monotonic and
$\frac{F^{\prime }\left( x\right) }{G\left( x\right) }\geq m>0$
or $\frac{F^{\prime }\left( x\right)}{G\left( x\right) }\leq -m<0.$ Then we have
\begin{equation*}
\left\vert \int_{a}^{b}G\left( x\right) e^{iF\left( x\right)
}dx\right\vert \leq 4m^{-1}.
\end{equation*}
\end{lem}

\noindent We will also use the following Van der Corput B-process (see \cite{Kuhleitner}, Lemma 2.2).

\begin{lem}\label{Convert formula}
Let $C_{i},i=1,\cdots ,7$ be absolute positive constants. Suppose that $g$ is a real-valued function which
has four continuous derivatives on the interval $[A,B].$ Let $L$ and $W$ be real parameters not less than $1,$
such that $C_{1}L\leq B-A\leq C_{2}L,$
\begin{equation*}
\left\vert g^{\left( j\right) }\left( \omega \right) \right\vert
\leq -C_{j+2}WL^{1-j},\ \mathrm{\ \ \ \ for}\ \omega \in \lbrack A,B],\text{
}j=1,2,3,4,
\end{equation*}
and
\begin{equation*}
g^{\prime \prime }\left( \omega \right) \geq C_{7}WL^{-1}\ \mathrm{\ \ or\ \ }\
g^{\prime \prime }\left( \omega \right) \leq -C_{7}WL^{-1},\ \mathrm{\ \ \ \ for\ }\
\omega \in \lbrack A,B].
\end{equation*}
Let $\phi $ denote the inverse function of $g^{\prime }.$ Define%
\begin{equation*}
\epsilon _{f}=\left\{
\begin{array}{cc}
e^{\frac{\pi i}{4}}, &\ \mathrm{\ if}\ g^{\prime \prime }\left( \omega \right) >0%
\ \mathrm{\ \ \ \ for\ }\ \omega \in \lbrack A,B], \\
e^{-\frac{\pi i}{4}}, &\ \mathrm{if}\ g^{\prime \prime }\left( \omega \right) <0%
\ \mathrm{\ \ \ \ for\ }\ \omega \in \lbrack A,B]
\end{array}
\right.
\end{equation*}
and
\begin{equation*}
r\left( x\right) =\left\{
\begin{array}{cc}
0, &\ \mathrm{if}\ g^{\prime }\left( x\right) \in\mathbb{Z}, \\
\min \left( \frac{1}{\left\vert \left\vert g^{\prime }\left(
x\right)
\right\vert \right\vert },\sqrt{\frac{L}{W}}\right) , &\ \mathrm{else,}\
\end{array}
\right.
\end{equation*}
with $\left\vert \left\vert \cdot \right\vert \right\vert $ denoting
the distance from the nearest integer. Then it follows that
\begin{eqnarray*}
\sum_{A< l\leq B}e\left( g\left( l\right) \right) &=&\epsilon
_{f}\sideset{}{^{\prime \prime}}\sum\limits_{\min \left( g^{\prime }\left( A\right) ,g^{\prime}
\left( B\right) \right) \leq k\leq \max \left( g^{\prime }\left(A\right) ,g^{\prime }\left( B\right) \right) }
\frac{e\left( g\left(\phi \left( k\right) \right) -k\phi \left( k\right) \right) }{\sqrt{
\left\vert g^{\prime \prime }\left( \phi \left( k\right) \right)
\right\vert}} \\
&&+O\left( r\left( A\right) +r\left( B\right) +\log \left(
2+W\right) \right) ,
\end{eqnarray*}
with the notation
\begin{equation*}
\sideset{}{^{\prime \prime
}}\sum\limits_{a\leq m\leq b}\Phi \left( n\right) =\frac{1}{2}%
\left( \chi_{\mathbb{Z}
}\left( a\right) \Phi \left( a\right) +\chi _{
\mathbb{Z}}\left( b\right) \Phi \left( b\right) \right)
+\sum\limits_{a<m<b}\Phi \left( n\right),
\end{equation*}
where $\chi _{\mathbb{Z}
}\left( \cdot \right) $ is the indicator function of the integers and
the $O$-constant depends on the constants $C_{i},i=1,\cdots ,7.$
\end{lem}

%\begin{lem}
%\label{mean value of Dirichlet polynomial}Let $a_{1},a_{2},\cdots
%a_{N}$ be complex numbers. Then we have
%\begin{equation*}
%\int_{T}^{2T}\left\vert a_{n}n^{-it}\right\vert ^{2}dt=T\sum_{n\leq
%N}\left\vert a_{n}\right\vert ^{2}+O\left( \sum_{n\leq N}n\left\vert
%a_{n}\right\vert ^{2}\right) .
%\end{equation*}
%\end{lem}

\bigskip

\section{Asymptotic formula for the summatory
function}

\begin{prop}\label{Asymptotic formula for D_ab}
Let $\alpha =\frac{p_{1}}{q_{1}}$ and $%
\beta =\frac{p_{2}}{q_{2}}$ with $p_{1},p_{2},q_{1},q_{2}\in\mathbb{N},$
$\gcd \left( p_{1},q_{1}\right)=1$ and $\gcd \left(q_{1},q_{2}\right) =1.$
We have
\begin{equation*}
D_{\alpha ,\beta }(x)=c_{1}x+c_{2}\sqrt{x}+\Delta _{\alpha ,\beta}(x),
\end{equation*}
where
\begin{equation*}
c_{1}=c_{1}\left( \alpha ,\beta \right) =\frac{\log \alpha -\log \beta
}{2},\ \ \ \ c_{2}=c_{2}\left( \alpha ,\beta \right) =\frac{1}{2}\left( \sqrt{\frac{1}{%
p_{2}q_{2}}}-\sqrt{\frac{1}{p_{1}q_{1}}}\right) ,
\end{equation*}
and
\begin{equation}
\Delta _{\alpha ,\beta }(x)=-\sum_{\sqrt{\frac{x}{\beta }}<l\leq \sqrt{\frac{%
x}{\alpha }}}\psi \left( \frac{x}{l}\right) +O_{\alpha ,\beta
}\left( 1\right) .  \label{Expression of Delta_ab}
\end{equation}
\end{prop}
\begin{proof}
It is enough to consider $d_{\alpha }(n)=\sum\limits_{\substack{n=kl\\k\leq \alpha l}}1$
and $D_{\alpha }(x)=\sum\limits_{n\leq x}d_{\alpha }(n)$.
Clearly,
\begin{eqnarray*}
D_{\alpha }(x)&=&\sum_{\substack{ kl\leq x \\ k\leq \alpha l}}1
=\sum_{l\leq x}\sum_{k\leq \min \left( x/l,\alpha l\right) }1.
\end{eqnarray*}
Write
\begin{eqnarray}
\label{Sum1+Sum2}
D_{\alpha }(x)=\sum\nolimits_{1}+\sum\nolimits_{2}\ ,
\end{eqnarray}
with
\begin{eqnarray*}
\sum\nolimits_{1}=\sum_{l\leq \sqrt{\frac{x}{\alpha }}}\sum_{k\leq \alpha l}1\ \ \ \ \mathrm{and}
\ \ \ \ \ \sum\nolimits_{2}=\sum_{\sqrt{
\frac{x}{\alpha }}<l\leq x}\sum_{k\leq x/l}1.
\end{eqnarray*}
It is easy to see that%
\begin{eqnarray}
\label{Sum 1 in the first lemma}
\sum\nolimits_{1}
%&=&\sum_{l\leq \sqrt{\frac{x}{\alpha }}}[\alpha l] \\
&=&\sum_{l\leq \sqrt{\frac{x}{%
\alpha }}}{\left(\alpha l-\psi \left( \alpha l\right)-1/2\right)} \nonumber\\
%&=&\frac{\alpha }{2}\left[ \sqrt{\frac{x}{\alpha }}\right] \left(
%\left[ \sqrt{\frac{x}{\alpha }}\right] +1\right) -\sum_{l\leq
%\sqrt{\frac{x}{\alpha
%}}}\psi \left( \alpha l\right) -\frac{1}{2}\left[ \sqrt{\frac{x}{\alpha }}%
%\right]  \\
&=&\frac{x}{2}-\sqrt{\alpha x}\psi \left( \sqrt{\frac{x}{\alpha
}}\right)
-\sum_{l\leq \sqrt{\frac{x}{\alpha }}}\psi \left( \alpha l\right) -\frac{1}{2%
}\sqrt{\frac{x}{\alpha }}+O_{\alpha }\left( 1\right).
\end{eqnarray}
Similarly,
\begin{eqnarray}
\label{Sum 2 in the first lemma}
\sum\nolimits_{2}
%&=&\sum_{\sqrt{\frac{x}{\alpha }}<l\leq x}\left[ x/l\right]  \\
&=&\sum_{\sqrt{\frac{x}{\alpha }}<l\leq x}\left( x/l-\psi
(x/l)-1/2\right)\nonumber
\\
%&=&x\sum_{\sqrt{\frac{x}{\alpha }}<l\leq x}1/l-\sum_{\sqrt{\frac{x}{\alpha }}%
%<l\leq x}\psi (x/l)-\frac{1}{2}\sum_{\sqrt{\frac{x}{\alpha }}<l\leq x}1 \\
&=&x\sum_{\sqrt{\frac{x}{\alpha }}<l\leq x}1/l-\sum_{\sqrt{\frac{x}{\alpha }}%
<l\leq x}\psi (x/l)-\frac{1}{2}x+\frac{1}{2}\sqrt{\frac{x}{\alpha
}}+O\left( 1\right) .
\end{eqnarray}%
By the Euler-Maclaurin summation, we have
\begin{eqnarray}
\label{sum 1/l}
\sum_{\sqrt{\frac{x}{\alpha }}<l\leq x}1/l
%&=&\int_{\sqrt{\frac{x}{\alpha }}}^{x}\frac{1}{t}dt
%-\frac{1}{x}\psi \left( x\right) +\sqrt{\frac{\alpha }{x}}
%\psi \left( \sqrt{\frac{x}{\alpha }}\right)\nonumber\\
%&&+\left( -\frac{1}{x^{2}}\right) \psi _{1}\left( x\right) -\left( -\frac{\alpha }{x}\right) \psi _{1}
%\left( \sqrt{\frac{x}{\alpha }}\right) -\int_{\sqrt{\frac{x}{\alpha }}}^{x}\left( \frac{2}{x^{3}}\right)
%\psi_{1}\left(u\right) dt \nonumber\\
%&=&\int_{\sqrt{\frac{x}{\alpha }}}^{x}\frac{1}{t}dt-\frac{1}{x}\psi
%\left(
%x\right) +\sqrt{\frac{\alpha }{x}}\psi \left( \sqrt{\frac{x}{\alpha }}%
%\right)  \\
%&&-\frac{1}{x^{2}}\psi _{1}\left( x\right) +\frac{\alpha }{x}\psi
%_{1}\left( \sqrt{\frac{x}{\alpha }}\right)
%-\int_{\sqrt{\frac{x}{\alpha }}}^{x}\left(
%\frac{2}{x^{3}}\right) \psi _{1}\left( u\right) dt \\
%&=&\int_{\sqrt{\frac{x}{\alpha }}}^{x}\frac{1}{t}dt+\sqrt{\frac{\alpha }{x}}%
%\psi \left( \sqrt{\frac{x}{\alpha }}\right) +O_{\alpha }\left( \frac{1}{x}%
%\right)  \\
%&=&\left( \log x-\frac{1}{2}\log \left( \frac{x}{\alpha }\right) \right) +%
%\sqrt{\frac{\alpha }{x}}\psi \left( \sqrt{\frac{x}{\alpha }}\right)
%+O_{\alpha }\left( \frac{1}{x}\right)  \\
&=&\frac{1}{2}\log x+\frac{1}{2}\log \alpha +\sqrt{\frac{\alpha
}{x}}\psi \left( \sqrt{\frac{x}{\alpha }}\right) +O_{a}\left(
\frac{1}{x}\right).
\end{eqnarray}%
Combining (\ref{Sum1+Sum2})-(\ref{sum 1/l}), we get%
\begin{eqnarray*}
D_{\alpha }(x)
%&=&\frac{x}{2}-\sqrt{\alpha x}\psi \left( \sqrt{%
%\frac{x}{\alpha }}\right) -\sum_{l\leq \sqrt{\frac{x}{\alpha }}}\psi
%\left( \alpha l\right) -\frac{1}{2}\sqrt{\frac{x}{\alpha
%}}+O_{\alpha }\left(
%1\right)  \\
%&&+x\left( \frac{1}{2}\log x+\frac{1}{2}\log \alpha +\sqrt{\frac{\alpha }{x}}%
%\psi \left( \sqrt{\frac{x}{\alpha }}\right) +O_{a}\left(
%\frac{1}{x}\right)
%\right)  \\
%&&-\sum_{\sqrt{\frac{x}{\alpha }}<l\leq x}\psi (x/l)-\frac{1}{2}x+\frac{1}{2}%
%\sqrt{\frac{x}{\alpha }}+O\left( 1\right)  \\
%&=&-\sum_{l\leq \sqrt{\frac{x}{\alpha }}}\psi \left( \alpha l\right) +\frac{x%
%}{2}\log x+\frac{x}{2}\log \alpha -\sum_{\sqrt{\frac{x}{\alpha
%}}<l\leq
%x}\psi (x/l)+O_{\alpha }\left( 1\right)  \\
=\frac{x}{2}\log x+\frac{\log \alpha }{2}x-\sum_{\sqrt{\frac{x}{\alpha }}%
<l\leq x}\psi (x/l)-\sum_{l\leq \sqrt{\frac{x}{\alpha }}}\psi \left(
\alpha l\right) +O_{\alpha }\left( 1\right) .
\end{eqnarray*}%
Note that
\begin{equation*}
-\sum_{l\leq \sqrt{\frac{x}{\alpha }}}\psi \left( \alpha l\right)
=-\sum_{l\leq \sqrt{\frac{q_1x}{p_1}}}\psi \left( \frac{p_1 l}{q_1} \right)
=\frac{1}{2}\sqrt{\frac{x}{p_1q_1}}+O_\alpha\left( 1\right).
\end{equation*}%
Hence
\begin{eqnarray}\label{D_alpha}
D_{\alpha }(x)
=\frac{x}{2}\log x+\frac{\log \alpha }{2}x-\sum_{\sqrt{\frac{x}{\alpha }}
<l\leq x}\psi (x/l)+\frac{1}{2}\sqrt{\frac{x}{p_1q_1}}+O_{\alpha }\left( 1\right) .
\end{eqnarray}
Similarly, for  $d_{\beta }(n)=\sum\limits_{\substack{n=kl\\k\leq \beta l}}1$
and $D_{\beta }(x)=\sum\limits_{n\leq x}d_{\beta }(n)$, we have
\begin{equation}\label{D_beta}
D_{\beta }(x)=\frac{x}{2}\log x+\frac{\log \beta }{2}x-\sum_{%
\sqrt{\frac{x}{\beta }}<l\leq x}\psi (x/l)+\frac{1}{2}\sqrt{\frac{x}{p_2q_2}}+O_{\beta }\left( 1\right).
\end{equation}
Now Proposition \ref{Asymptotic formula for D_ab} follows from (\ref{D_alpha}), (\ref{D_beta}) and
\begin{equation*}
D_{\alpha ,\beta }(x)=D_{\beta }(x)-D_{\alpha }(x).
\end{equation*}
\end{proof}
\begin{corollary}\label{Upper bound of Delta_ab}
We have
\begin{equation*}
D_{\alpha ,\beta }(x)=c_{1}x+c_{2}\sqrt{x}+O_{\alpha ,\beta }\left( x^{\frac{1}{3}}\right),
\end{equation*}
where $c_{1},\ c_{2}$ are the same as Proposition {\rm \ref{Asymptotic formula for D_ab}}.
\end{corollary}

\begin{proof}
This can be proved easily (even with a better upper bound for the
error term) by applying Lemma \ref{Fourier expansion of Psi} and exponential pairs (see \cite{Graham and Kolesnik})
to Proposition \ref{Asymptotic formula for D_ab}.
\end{proof}

%%%%%%%%%%%%%%%%%%%%%%%%%%%%%%%%%%%%%%%%%%%%%%%%%%%%%%%%%%%%%%%
%%%%%%%%%%%%%%%%%%%%%%%%%%%%%%%%%%%%%%%%%%%%%%%%%%%%%%%%%%%%%%%
%%%%%%%%%%%%%%%%%%%%%%%%%%%%%%%%%%%%%%%%%%%%%%%%%%%%%%%%%%%%%%%

\bigskip
\section{A Voronoi type formula}

Define
\begin{equation*}
d_{\alpha ,\beta }\left( n,H\right) =\underset{n=hk}{\sum_{1\leq
h\leq H}\sideset{}{^{\prime \prime
}}\sum\limits_{\ h\alpha \leq k\leq h\beta }}1,
\end{equation*}
where the notation $\sideset{}{^{\prime \prime}}\sum\limits$ is the same as Lemma \ref{Convert formula}.
Using the Van der Corput B-process and the same argument as Section 6.2 of \cite{Zhai},
we can derive the following Voronoi type formula for $\Delta _{\protect\alpha,\protect\beta }(x)$.

\begin{lem}\label{Voronoi formula}
For any $H\geq 2$ and rational numbers $0<\alpha<\beta$, we have
\begin{equation*}
\Delta _{\alpha ,\beta }(x)=M_{\alpha ,\beta }\left( x,H\right)
+E_{\alpha ,\beta }\left( x,H\right) +F_{\alpha ,\beta }\left(x,H\right) ,
\end{equation*}
where
\begin{equation}\label{M_ab(x,H)}
M_{\alpha ,\beta }\left( x,H\right)=\frac{x^{\frac{1}{4}}}{\pi \sqrt{2}}\sum_{n\leq \beta H^{2}}
\frac{d_{\alpha ,\beta }\left( n,H\right) }{n^{\frac{3}{4}}}\cos \left( 4\pi \sqrt{nx}-\frac{\pi }{4}\right),
\end{equation}
\begin{equation}
E_{\alpha ,\beta }\left( x,H\right) \ll \sum_{\sqrt{\frac{x}{\alpha
}}<l\leq \sqrt{\frac{x}{\beta }}}G\left( \frac{x}{l},H\right)
\end{equation}
and
\begin{equation}\label{bound for F(x,H)}
F_{\alpha ,\beta }\left( x,H\right) \ll _{\alpha ,\beta }\log H.
\end{equation}
\end{lem}
\begin{proof}
Applying Lemma \ref{Fourier expansion of Psi} to (\ref{Expression of Delta_ab}), we get
\begin{equation*}
\Delta _{\alpha ,\beta }(x)=\frac{1}{2\pi i}\sum_{1\leq \left\vert
h\right\vert \leq H}\frac{1}{h}\sum_{\sqrt{\frac{x}{\beta }}<l\leq \sqrt{\frac{x}{\alpha }}}e\left( \frac{hx}{l}\right)
+E_{\alpha ,\beta }\left( x,H\right)+O_{\alpha,\beta }\left(1\right),
\end{equation*}
with
\begin{equation}\label{E_ab <<}
E_{\alpha ,\beta }\left( x,H\right) \ll \sum_{\sqrt{\frac{x}{\alpha}}<l\leq
\sqrt{\frac{x}{\beta }}}G\left( \frac{x}{l},H\right).
\end{equation}
Let
\begin{equation}\label{Sigma over h and l}
S_{\alpha,\beta}(x,H) =\frac{1}{2\pi i}\sum_{1\leq h\leq H}\frac{1}{h}\sum_{\sqrt{\frac{x}{
\beta }}<l\leq \sqrt{\frac{x}{\alpha }}}e\left( \frac{hx}{l}\right),
\end{equation}
then we can write
\begin{equation}\label{Delta_ab=M+E+F}
\Delta _{\alpha ,\beta }(x)=\frac{1}{2\pi i}\left( S_{\alpha,\beta}(x,H) -\overline{S_{\alpha,\beta}(x,H)}
\right) +E_{\alpha ,\beta }\left( x,H\right) +O_{\alpha ,\beta}\left(1\right).
\end{equation}
To treat the inner sum
\begin{equation*}
\sum\limits_{\sqrt{\frac{x}{\beta }}<l\leq \sqrt{\frac{x}{\alpha
}}}e\left( \frac{hx}{l}\right) \text{ \ for }1\leq h\leq H
\end{equation*}
in (\ref{Sigma over h and l}), we apply Lemma \ref{Convert formula}. Let
\begin{equation*}
A=\sqrt{\frac{x}{\beta }},\ \ B=\sqrt{\frac{x}{\alpha }}\ \ {\rm{and}}
\ \ g\left( l\right) =\frac{hx}{l},
\end{equation*}
then we have
\begin{equation*}
g^{\prime }\left( l\right) =-\frac{hx}{l^{2}},\ \ g^{\prime
\prime }\left( l\right) =\frac{2hx}{l^{3}},\ \ g^{\left(
3\right) }\left(
l\right) =-\frac{6hx}{l^{4}},\ \ g^{\left( 4\right) }\left( l\right) =\frac{24hx}{l^{5}},
\end{equation*}
\begin{equation*}
g^{\prime }\left( B\right) =-h\alpha ,\ \ g^{\prime }\left(
A\right) =-h\beta,\ \ \frac{2\alpha ^{\frac{3}{2}}h}{\sqrt{x}}<g^{\prime \prime }\left(
l\right) \leq \frac{2\beta ^{\frac{3}{2}}h}{\sqrt{x}}
\end{equation*}
and
\begin{equation*}
\left\vert g^{\prime \prime \prime }\left( l\right) \right\vert \ll
_{\alpha ,\beta }\frac{h}{x}.
\end{equation*}
Hence we can take
\begin{equation*}
W=1,\ \ L=\frac{\sqrt{x}}{h},\ \ \phi \left( k\right) =\sqrt{-\frac{hx}{k}},
\end{equation*}
\begin{equation*}
g\left( \phi
\left( k\right) \right) -k\phi \left( k\right) =2\sqrt{-hkx},\ \ {\rm{and}}\ \
g^{\prime \prime }\left( \phi \left( k\right) \right)
=2\sqrt{\frac{\left( -k\right) ^{3}}{hx}}.
\end{equation*}
Noting $\alpha ,\beta $ are rational numbers, we have
\begin{equation}\label{r(A) and r(B)}
r\left( A\right) ,r\left( B\right) \ll _{\alpha ,\beta }1.
\end{equation}
Now for $1\leq h\leq H,$ by Lemma \ref{Convert formula}, we get
\begin{eqnarray}\label{Sum over l e()}
\sum\limits_{\sqrt{\frac{x}{\alpha }}<l\leq \sqrt{\frac{x}{\beta
}}}e\left(
\frac{hx}{l}\right) &=&\frac{e^{\frac{\pi i}{4}}}{\sqrt{2}}
\sideset{}{^{\prime \prime
}}\sum\limits_{\ -h\beta \leq k\leq -h\alpha }\frac{h^{\frac{1}{
4}}x^{\frac{1}{4}}}{\left( -k\right) ^{\frac{3}{4}}}e\left( 2\sqrt{-hkx}
\right) +O_{\alpha ,\beta }\left( 1\right)  \\
&=&\frac{1}{\sqrt{2}}\sideset{}{^{\prime \prime
}}\sum\limits_{\ h\alpha \leq k\leq h\beta
}\frac{h^{\frac{1}{4}}x^{\frac{1}{4}}}{k^{\frac{3}{4}}}e\left( 2\sqrt{hkx}+
\frac{1}{8}\right) +O_{\alpha ,\beta }\left( 1\right) .\text{ }
\notag
\end{eqnarray}
Inserting (\ref{Sum over l e()}) to (\ref{Sigma over h and l}) gives
\begin{eqnarray*}
S_{\alpha,\beta}(x,H) &=&\frac{1}{\sqrt{2}}\sum_{1\leq h\leq H}\frac{1}{h}
\sideset{}{^{\prime \prime}}\sum\limits_{\ h\alpha \leq k\leq h\beta }\frac{h^{\frac{1}{4}}
x^{\frac{1}{4}}}{k^{\frac{3}{4}}}e\left( 2\sqrt{hkx}+\frac{1}{8}\right)+O_{\alpha ,\beta}\left( \log H\right)\\
&=&\frac{x^{\frac{1}{4}}}{\sqrt{2}}\sum_{1\leq h\leq H}\sideset{}{^{\prime \prime}}\sum\limits_{\ h\alpha
\leq k\leq h\beta }\frac{1}{\left( hk\right) ^{\frac{3}{4}}}e\left( 2\sqrt{hkx}+\frac{1}{8}\right) +O_{\alpha ,\beta }
\left(\log H\right) \\
&=&\frac{x^{\frac{1}{4}}}{\sqrt{2}}\sum_{n\leq \beta
H^{2}}\frac{d_{\alpha
,\beta }\left( n,H\right) }{n^{\frac{3}{4}}}e\left( 2\sqrt{nx}+\frac{1}{8}%
\right) +O_{\alpha ,\beta }\left( \log H\right) .
\end{eqnarray*}
Thus%
\begin{equation*}
\frac{1}{2\pi i}\left( S_{\alpha,\beta}(x,H) -\overline{S_{\alpha,\beta}(x,H) }\right) =\frac{x^{\frac{1}{4}}}{
\pi \sqrt{2}}\sum_{n\leq \beta H^{2}}\frac{d_{\alpha ,\beta }\left(
n,H\right) }{n^{\frac{3}{4}}}\cos \left( 4\pi \sqrt{nx}-\frac{\pi }{4}
\right) +O_{\alpha ,\beta }\left( \log H\right).
\end{equation*}
This combining with (\ref{Delta_ab=M+E+F}) and (\ref{E_ab <<}) yields Lemma \ref{Voronoi formula}.
\end{proof}
\begin{rem}
The bound (\ref{bound for F(x,H)}) is important in the proof of Theorem \ref{Main theorem}. If $\alpha,\beta$ are not rational numbers, the author can't get the estimate (\ref{bound for F(x,H)}).
Because in this case the estimate (\ref{r(A) and r(B)}) does not holds.
\end{rem}

%%%%%%%%%%%%%%%%%%%%%%%%%%%%%%%%%%%%%%%%%%%%%%%%%%%%%%%%%%%%%%%
%%%%%%%%%%%%%%%%%%%%%%%%%%%%%%%%%%%%%%%%%%%%%%%%%%%%%%%%%%%%%%%
%%%%%%%%%%%%%%%%%%%%%%%%%%%%%%%%%%%%%%%%%%%%%%%%%%%%%%%%%%%%%%%
\section{Proof of Proposition \ref{Proposition of Approximation}}

First, let's show that $\zeta _{\alpha ,\beta }(s)$ can be analyticly
extended to $\sigma >\frac{1}{3}.$ For $\sigma >1$ and any $N\geq
1,$ write

\begin{eqnarray*}
\zeta _{\alpha ,\beta }(s) &=&\sum_{n\leq N}\frac{d_{\alpha ,\beta }(n)}{n^{s}}
+\sum_{n>N}\frac{d_{\alpha ,\beta }(n)}{n^{s}} \\
&=&\sum_{n\leq N}\frac{d_{\alpha ,\beta}(n)}{n^{s}}+\int_{N^{+}}^{\infty }u^{-s}dD_{\alpha ,\beta}(u),
\end{eqnarray*}
where $D_{\alpha ,\beta }(u)$ is defined by (\ref{Summatory function of d_ab(n)}).
Applying Proposition \ref{Asymptotic formula for D_ab}, we get

\begin{eqnarray*}
\zeta _{\alpha ,\beta }(s) &=&\sum_{n\leq N}\frac{d_{\alpha ,\beta }(n)}{n^{s}}
+\int_{N^{+}}^{\infty }u^{-s}d\left(c_{1}u+c_{2}\sqrt{u}+\Delta_{\alpha ,\beta }(u)\right) \\
&=&\sum_{n\leq N}\frac{d_{\alpha ,\beta }(n)}{n^{s}}+c_{1}\int_{N^{+}}^{\infty }u^{-s}du
+\frac{c_{2}}{2}\int_{N^{+}}^{\infty}u^{-s-1/2}du+\int_{N^{+}}^{\infty }u^{-s}d\Delta _{\alpha ,\beta}(u).
\end{eqnarray*}
By partial integration, we have
\begin{equation}\label{Zeta_ab  break 1}
\zeta _{\alpha ,\beta }(s)=\sum_{n\leq N}\frac{d_{\alpha ,\beta }(n)}{n^{s}}
-\frac{c_{1}N^{1-s}}{1-s}-\frac{c_{2}N^{\frac{1}{2}-s}}{1-2s}
+s\int_{N^{+}}^{\infty }\Delta _{\alpha ,\beta }(u)u^{-s-1}du+O\left( N^{\frac{1}{3}-\sigma }\right) .
\end{equation}
From Corollary \ref{Upper bound of Delta_ab}, we can see that the
integral in (\ref{Zeta_ab  break 1}) is absolutely convergent for $\sigma>\frac{1}{3},$
hence (\ref{Zeta_ab  break 1}) gives an analytic continuation of $\zeta _{\alpha ,\beta }(s)$ for
$\sigma >\frac{1}{3}.$ This proves the first assertion of Proposition \ref{Proposition of Approximation}.

Now suppose $\sigma >\frac{1}{3}$ and $2\leq T<t\leq 2T$. From now on, we take $N=T^{A}$
with $A>0$ being a constant, sufficiently large. Break the sum in (\ref{Zeta_ab break 1}) into
\begin{equation}\label{Zeta_ab break 2}
\sum_{n\leq N}\frac{d_{\alpha ,\beta }(n)}{n^{s}}
=\sum_{n\leq x}\frac{d_{\alpha ,\beta }(n)}{n^{s}}+\sum_{x<n\leq N}\frac{d_{\alpha ,\beta }(n)}{n^{s}}.
\end{equation}
For the second sum, applying Proposition \ref{Asymptotic formula for D_ab}
again,
we have%
\begin{eqnarray*}
\sum_{x<n\leq N}\frac{d_{\alpha ,\beta }(n)}{n^{s}}
&=&\int_{x}^{N}u^{-s}dD_{\alpha ,\beta }(u) \\
&=&\int_{x}^{N}u^{-s}d\left( c_{1}\left( \alpha ,\beta \right)u
+c_{2}\left( \alpha ,\beta \right) \sqrt{u}+\Delta _{\alpha ,\beta}(u)\right) .
\end{eqnarray*}
By partial integration, we have
\begin{eqnarray}\label{Zeta_ab break 33}
\sum_{x<n\leq N}\frac{d_{\alpha ,\beta }(n)}{n^{s}} &=&c_{1}\left(
\alpha ,\beta \right) \left(
\frac{N^{1-s}}{1-s}-\frac{x^{1-s}}{1-s}\right)+c_{2}\left( \alpha ,\beta \right)
\left( \frac{N^{1/2-s}}{1-2s}-\frac{x^{1/2-s}}{1-2s}\right)\\
&+&N^{-s}\Delta _{\alpha ,\beta }(N)-x^{-s}\Delta _{\alpha ,\beta}(x)+s\int_{x}^{N}\Delta _{\alpha ,\beta }(u)u^{-s-1}du.\notag
\end{eqnarray}
Combining (\ref{Zeta_ab break 1}), (\ref{Zeta_ab break 2}) and
(\ref{Zeta_ab
break 33}), we get
\begin{equation}\label{Zeta_ab break 4}
\zeta _{\alpha ,\beta }(s)=\sum_{n\leq x}\frac{d_{\alpha ,\beta }(n)}{n^{s}}
+s\int_{x}^{N}\Delta _{\alpha ,\beta }(u)u^{-s-1}du+O_{\alpha ,\beta ,\sigma}\left( x^{1-\sigma }t^{-1}+x^{1/3-\sigma }\right)
\end{equation}
holds for any $\sigma >\frac{1}{3}.$

Our tool to prove Proposition \ref{Proposition of Approximation} is the Voronoi formula for $\Delta _{\alpha ,\beta }(x).$
Using Lemma \ref{Voronoi formula}, we can write

\begin{equation}\label{Start point. Break Zeta_ab(s)}
s\int_{x}^{N}\Delta _{\alpha ,\beta }(u)u^{-s-1}du=\mathfrak{M}\left( s\right)
+\mathfrak{E}\left( s\right)+\mathfrak{F}\left( s\right),
\end{equation}
where
\begin{equation}\label{M(s)}
\mathfrak{M}\left( s\right) =\mathfrak{M}_{\alpha ,\beta }
\left(s,H,x,N\right) =s\int_{x}^{N}M_{\alpha ,\beta }\left( u,H\right)u^{-s-1}du,
\end{equation}
\begin{equation}\label{E(s)}
\mathfrak{E}\left( s\right) =\mathfrak{E}_{\alpha ,\beta }
\left(s,H,x,N\right) =s\int_{x}^{N}E_{\alpha ,\beta }\left( u,H\right)u^{-s-1}du,
\end{equation}
and
\begin{equation}\label{F(s)}
\mathfrak{F}\left( s\right) =\mathfrak{F}_{\alpha ,\beta }\left(s,H,x,N\right)
=s\int_{x}^{N}F_{\alpha ,\beta }\left( u,H\right)u^{-s-1}du.
\end{equation}

In Section 7, we will show that the upper bound of $\mathfrak{E}\left( s\right)$ is small, when $H$ is large comparing to $N$
and the mean square of $\mathfrak{F}\left( s\right)$ has an acceptable estimate; see Lemma \ref{Lemma Flower F(s)} and Lemma \ref{Lemma Flower E(s)}, respectively. In Section 8, we will pick out the second term in (\ref{Approximation for Zeta_ab}) from $\mathfrak{M}(s)$; see Lemma \ref{Lemma Flower M(s)}, Combining (\ref{Zeta_ab break 4}) with Lemmas \ref{Lemma Flower F(s)}-\ref{Approximation for Zeta_ab}, we get Proposition \ref{Proposition of Approximation}.

%%%%%%%%%%%%%%%%%%%%%%%%%%%%%%%%%%%%%%%%%%%%%%%%%%%%%%%%%%%%%%%
%%%%%%%%%%%%%%%%%%%%%%%%%%%%%%%%%%%%%%%%%%%%%%%%%%%%%%%%%%%%%%%
%%%%%%%%%%%%%%%%%%%%%%%%%%%%%%%%%%%%%%%%%%%%%%%%%%%%%%%%%%%%%%%

\bigskip

\section{An upper bound and a mean square estimate}
To bound $\mathfrak{E}\left( s\right),$ we need the following mean value estimate
for $G\left( u,H\right) $ defined by (\ref{Def. of G(u,H)}).
\begin{lem}\label{mean value of G(u,H)}
For any $N\geq 1$ and $H\geq 2,$ we have
\begin{equation*}
\int_{0}^{N}G\left( u,H\right) du\ll \frac{N\log H}{H}.
\end{equation*}
\end{lem}
\begin{proof}
Note that $G\left( u,H\right) $ is a positive peridodic function
with period
$1,$ then we have
\begin{eqnarray*}
\int_{0}^{N}G\left( u,H\right) du &\leq &\sum_{k=0}^{\left[ N\right]}\int_{k}^{k+1}G\left( u,H\right) du \\
&\ll&N\int_{0}^{1}G\left( u,H\right) du \\
&=&N\int_{-1/2}^{1/2}\min \left( 1,\frac{1}{H\left\vert \left\vert u\right\vert \right\vert }\right) du.
\end{eqnarray*}
Noting $\left\vert \left\vert u\right\vert \right\vert =\left\vert
u\right\vert $ for $u\in \lbrack -1/2,1/2],$ we get
\begin{eqnarray*}
\int_{0}^{N}G\left( u,H\right) du &\ll &N\int_{-1/2}^{1/2}\min \left( 1,\frac{1}{H\left\vert u\right\vert }\right) du \\
&\ll&N\int_{0}^{1/2}\min \left( 1,\frac{1}{Hu}\right) du \\
&\ll&N\int_{0}^{1/H}du+\frac{N}{H}\int_{1/H}^{1/2}\frac{1}{u}du,
\end{eqnarray*}
which yields
\begin{equation*}
\int_{0}^{N}G\left( u,H\right) du\ll \frac{N\log H}{H}.
\end{equation*}
\end{proof}

\noindent By Lemma \ref{mean value of G(u,H)}, we can get
\begin{lem}\label{Lemma Flower F(s)}
For any $\sigma >1/2$, we have
\begin{equation*}
\mathfrak{E}\left( s\right) \ll \frac{tx^{-\sigma -1}N^{2}\log
H}{H}.
\end{equation*}
\end{lem}

\begin{proof}
By (\ref{E(s)}) and trivial estimates, we get
\begin{eqnarray*}
\mathfrak{E}\left(s\right) &\ll &t\int_{x}^{N}\sum_{\sqrt{\frac{u}{\alpha}}
<l\leq \sqrt{\frac{u}{\beta }}}G\left( \frac{u}{l},H\right)u^{-\sigma -1}du\\
&\ll&tx^{-\sigma -1}\sum_{\sqrt{\frac{x}{\alpha }}<l\leq \sqrt{\frac{N}{\beta }}}\int_{x}^{N}G\left( \frac{u}{l},H\right) du \\
&=&tx^{-\sigma -1}\sum_{\sqrt{\frac{x}{\alpha }}<l\leq \sqrt{\frac{N}{\beta }}}l\int_{\frac{x}{l}}^{\frac{N}{l}}G\left( u,H\right) du.
\end{eqnarray*}
This combining with Lemma \ref{mean value of G(u,H)} yields
\begin{equation*}
\mathfrak{E}\left( s\right) \ll tx^{-\sigma -1}\sum_{l\leq \sqrt{\frac{N}{\beta}}}l\int_{0}^{N}G\left( u,H\right) du
\ll \frac{tx^{-\sigma-1}N^{2}\log H}{H}.
\end{equation*}
\end{proof}

Now we consider the mean square of $\mathfrak{F}\left( s\right)$.
\begin{lem}\label{Lemma Flower E(s)}
For $\sigma >1/2,$ we
have
\begin{equation*}
\int_{T}^{2T}\left\vert \mathfrak{F}\left( s\right) \right\vert^{2}dt\ll_{\alpha ,\beta,\sigma} x^{-2\sigma }T^{2}\log ^{2}H\log N.
\end{equation*}
\end{lem}

\begin{proof}
Noting $F_{\alpha ,\beta }\left( u\right) \ll _{\alpha ,\beta }\log
H$ and
unfolding the square in the integral, we get
\begin{eqnarray*}
\int_{T}^{2T}\left\vert \mathfrak{F}\left( s\right) \right\vert^{2}dt &\ll &
T^{2}\int_{T}^{2T}\left\vert \int_{x}^{N}F_{\alpha,\beta }\left( u\right)u^{-s-1}du\right\vert ^{2}dt\\
&\ll_{\alpha ,\beta } &T^{2}\log ^{2}H\int_{x}^{N}\int_{x}^{N}\left(u_{1}u_{2}\right) ^{-\sigma -1}\left\vert \int_{T}^{2T}\left( \frac{u_{2}}{u_{1}}\right) ^{it}dt\right\vert du_{1}du_{2}
\end{eqnarray*}
Applying Lemma \ref{First drivitive test} to the above integral over $t$,
we have
\begin{eqnarray*}
\int_{T}^{2T}\left\vert \mathfrak{F}\left( s\right) \right\vert^{2}dt &\ll_{\alpha ,\beta }&
T^{2}\log^{2}H\int_{x}^{N}\int_{x}^{N}\left(u_{1}u_{2}\right)^{-\sigma -1}
\min \left( T,\frac{1}{\left\vert \log \frac{u_{2}}{u_{1}}\right\vert }\right) du_{1}du_{2} \\
&\ll_{\alpha ,\beta } &T^{2}\log^{2}H\int_{x}^{N}\int_{u_{1}}^{N}\left(u_{1}u_{2}\right)^{-\sigma -1}
\min \left( T,\frac{1}{\log \frac{u_{2}}{u_{1}}}\right) du_{1}du_{2}.
\end{eqnarray*}
Write this as
\begin{equation}\label{Meansquare of E(s)}
\int_{T}^{2T}\left\vert \mathfrak{F}\left( s\right) \right\vert^{2}dt
\ll _{\alpha ,\beta }\int_{1}+\int_{2}+\int_{3},
\end{equation}
where
\begin{equation*}
\int_{1}=T^{3}\log ^{2}H\int_{x}^{N}u_{1}^{-\sigma -1}\int_{u_{1}}^{e^{\frac{1}{T}}u_{1}}u_{2}^{-\sigma -1}du_{2}du_{1},
\end{equation*}
\begin{equation*}
\int_{2}=T^{2}\log ^{2}H\int_{x}^{N}u_{1}^{-\sigma -1}\int_{e^{\frac{1}{T}}u_{1}}^{\frac{3}{2}u_{1}}u_{2}^{-\sigma -1}
\frac{1}{\log \left( \frac{u_{2}}{u_{1}}\right) }du_{2}du_{1}
\end{equation*}
and
\begin{equation*}
\int_{3}=T^{2}\log ^{2}H\int_{x}^{N}u_{1}^{-\sigma -1}\int_{\frac{3}{2}u_{1}}^{N}u_{2}^{-\sigma -1}
\frac{1}{\log\frac{u_{2}}{u_{1}}}du_{2}du_{1}.
\end{equation*}
Let's deal with $\int_{i},\ i=1,2,3$ respectively. For $\int_{1},$ we have
\begin{eqnarray*}
\int_{1} &\ll &T^{3}\log ^{2}H\int_{x}^{N}u_{1}^{-2\sigma-2}\int_{u_{1}}^{e^{\frac{1}{T}}u_{1}}du_{2}du_{1} \\
&\ll &T^{3}\log ^{2}H\int_{x}^{N}u_{1}^{-2\sigma -2}\left( e^{\frac{1}{T}}u_{1}-u_{1}\right) du_{1} \\
&\ll &T^{3}\left( e^{\frac{1}{T}}-1\right) \log^{2}H\int_{x}^{N}u_{1}^{-2\sigma -1}du_{1},
\end{eqnarray*}
which yields
\begin{equation}\label{Meansquare of E(s) ingegral 1}
\int_{1}\ll _{\sigma }x^{-2\sigma }T^{2}\log ^{2}H.
\end{equation}
For $\int_{2},$ we have
\begin{eqnarray*}
\int_{2} &=&T^{2}\log ^{2}H\int_{x}^{N}u_{1}^{-\sigma -1}
\int_{e^{\frac{1}{T}}u_{1}}^{\frac{3}{2}u_{1}}u_{2}^{-\sigma -1}\frac{1}{\log \left( \frac{u_{2}}{u_{1}}\right)}du_{2}du_{1} \\
&=&T^{2}\log ^{2}H\int_{x}^{N}u_{1}^{-\sigma -1}\int_{e^{\frac{1}{T}}u_{1}}^{\frac{3}{2}u_{1}}u_{2}^{-\sigma -1}\frac{1}{\log \left(1+\frac{u_{2}-u_{1}}{u_{1}}\right) }du_{2}du_{1} \\
&\ll &T^{2}\log ^{2}H\int_{x}^{N}u_{1}^{-2\sigma -1}\int_{e^{\frac{1}{T}}u_{1}}^{\frac{3}{2}u_{1}}\frac{1}{u_{2}-u_{1}}du_{2}du_{1} \\
&\ll &T^{2}\log ^{2}H\int_{x}^{N}u_{1}^{-2\sigma -1}\log u_{1}du_{1},
\end{eqnarray*}
which yields
\begin{equation}
\int_{2}\ll _{\sigma }x^{-2\sigma }T^{2}\log ^{2}H\log N.
\label{Meansquare of E(s) ingegral 2}
\end{equation}
For $\int_{3},$ we have
\begin{equation}\label{Meansquare of E(s) ingegral 3}
\int_{3}\ll T^{2}\log ^{2}H\left( \int_{x}^{N}u^{-\sigma-1}du\right) ^{2}\ll _{\sigma }x^{-2\sigma }T^{2}\log ^{2}H.
\end{equation}
From (\ref{Meansquare of E(s)})-(\ref{Meansquare of E(s)
ingegral 3}), we get Lemma \ref{Lemma Flower E(s)}.
\end{proof}

%%%%%%%%%%%%%%%%%%%%%%%%%%%%%%%%%%%%%%%%%%%%%%%%%%%%%%%%%%%%%%%
%%%%%%%%%%%%%%%%%%%%%%%%%%%%%%%%%%%%%%%%%%%%%%%%%%%%%%%%%%%%%%%
%%%%%%%%%%%%%%%%%%%%%%%%%%%%%%%%%%%%%%%%%%%%%%%%%%%%%%%%%%%%%%%

\bigskip

\section{Picking out the second term in Proposition \ref{Proposition of Approximation}}

The second term in Proposition \ref{Proposition of Approximation} is hidden in $\mathfrak{M}\left( s\right)$. In this Section,
we will pick it out and prove
\begin{lem}\label{Lemma Flower M(s)}
For $\sigma>1/2,$ we have
\begin{eqnarray}\label{approximation for flower M}
\mathfrak{M}\left( s\right)&=&\chi ^{2}\left( s\right) \sum_{n\leq
x}d_{\alpha ,\beta }\left( n\right) n^{s-1}\\
\notag&&\indent+O\left( t^{-\frac{1}{2}%
}x^{1-\sigma }\log H+x^{1/2-\sigma }\log H+x^{\frac{1}{2}-\sigma
}\log t+x^{-\sigma }t^{\frac{1}{2}}\log t\right) .
\end{eqnarray}
\end{lem}

The idea of the proof for Lemma \ref{Lemma Flower M(s)} comes from Chapter 4 of \cite{Ivic}. By (\ref{M(s)}) and (\ref{M_ab(x,H)}), we have
\begin{equation*}
\mathfrak{M}\left( s\right) \mathfrak{=}\frac{s}{\pi \sqrt{2}}\int_{x}^{N}u^{-s-\frac{3}{4}}
\sum\limits_{n\leq \beta H^{2}}\frac{d_{\alpha,\beta }\left( n,H\right)}{n^{\frac{3}{4}}}
\cos \left( 4\pi \sqrt{nu}-\frac{\pi }{4}\right) du.
\end{equation*}
Let $\eta >0$ be a fixed, sufficiently small constant. Using $\cos z=\frac{e^{iz}+e^{-iz}}{2},$ we can write
\begin{equation}\label{Flower M(s)}
\mathfrak{M}\left( s\right) \mathfrak{=M}_{1}\left( s\right)+\mathfrak{M}_{2}\left( s\right)
+\mathfrak{M}_{3}\left( s\right)+\mathfrak{M}_{4}\left( s\right)
\end{equation}
with
\begin{equation*}
\mathfrak{M}_{1}\left( s\right) =\frac{s}{2\pi \sqrt{2}}\int_{x}^{N}u^{-s-\frac{3}{4}}
\sum\limits_{n\leq \left( 1+\eta \right)y}\frac{d_{\alpha,\beta }\left( n,H\right) }{n^{\frac{3}{4}}}e\left( 2\sqrt{nu}-\frac{1}{8}
\right) du,
\end{equation*}
\begin{equation*}
\mathfrak{M}_{2}\left( s\right) =\frac{s}{2\pi \sqrt{2}}\int_{x}^{N}u^{-s-%
\frac{3}{4}}\sum\limits_{\left( 1+\eta \right) y<n\leq \beta H^{2}}\frac{%
d_{\alpha ,\beta }\left( n,H\right) }{n^{\frac{3}{4}}}e\left( 2\sqrt{nu}-
\frac{1}{8}\right) du,
\end{equation*}
\begin{equation*}
\mathfrak{M}_{3}\left( s\right) =\frac{s}{2\pi \sqrt{2}}\int_{x}^{N}u^{-s-
\frac{3}{4}}\sum\limits_{n\leq y}\frac{d_{\alpha ,\beta }\left( n,H\right) }{
n^{\frac{3}{4}}}e\left( -2\sqrt{nu}+\frac{1}{8}\right) du
\end{equation*}
and
\begin{equation*}
\mathfrak{M}_{4}\left( s\right) =\frac{s}{2\pi \sqrt{2}}\int_{x}^{N}u^{-s-%
\frac{3}{4}}\sum\limits_{y<n\leq \beta H^{2}}\frac{d_{\alpha ,\beta
}\left( n,H\right) }{n^{\frac{3}{4}}}e\left(
-2\sqrt{nu}+\frac{1}{8}\right) du.
\end{equation*}
We will bound $\mathfrak{M}_{2}\left( s\right),$ $\mathfrak{M}_{3}\left( s\right)$
and $\mathfrak{M}_{4}\left( s\right)$ in the following Lemmas \ref{Flower M_2(s)}-\ref{Flower M_4(s)} and
pick out the first term on the right side of (\ref{approximation for flower M}) in Lemma \ref{Lemma Flower M_1}.
From Lemmas \ref{Flower M_2(s)}-\ref{Lemma Flower M_1} and (\ref{Flower M(s)}).
\begin{lem}\label{Flower M_2(s)}
For $\sigma >1/2,$ we have
\begin{equation*}
\mathfrak{M}_{2}\left( s\right) \ll t^{-\frac{1}{2}}x^{1-\sigma
}\log H.
\end{equation*}
\end{lem}

\begin{proof}
\bigskip Write
\begin{eqnarray*}
\mathfrak{M}_{2}\left( s\right) =\frac{s}{2\pi \sqrt{2}}\sum_{\left(1+\eta\right) y<n\leq \beta H^{2}}
\frac{d_{\alpha ,\beta }\left( n;H\right) }{n^{3/4}}\int_{x}^{N}u^{-\sigma -3/4}e\left( -\frac{t}{2\pi }\log u
+2\sqrt{nu}-1/8\right) du.
\end{eqnarray*}
In Lemma \ref{First drivitive test}, taking
\begin{equation*}
G\left( u\right) =u^{-\sigma -3/4}
\end{equation*}
and
\begin{equation*}
F_{t}\left( u\right) =-\frac{t}{2\pi }\log u+2\sqrt{nu}-1/8,
\end{equation*}
we have
\begin{equation*}
F_{t}^{\prime }\left( u\right) =-\frac{t}{2\pi u}+\sqrt{\frac{n}{u}}
\end{equation*}
and
\begin{equation*}
\frac{F_{t}^{\prime }\left( u\right) }{G\left( u\right) }=-\frac{t}{2\pi }u^{\sigma -1/4}+\sqrt{n}u^{\sigma +1/4}.
\end{equation*}
Since $n>\left(1+\eta\right)y,\ u>x$ and $4\pi ^{2}xy=t^{2},$ then
\begin{equation*}
\left( \frac{F_{t}^{\prime }\left( u\right) }{G\left( u\right)}\right) ^{\prime }
=-\left( \sigma -1/4\right) \frac{t}{2\pi}u^{\sigma -5/4}+\left( \sigma +1/4\right) \sqrt{n}u^{\sigma-3/4}>0.
\end{equation*}
Thus $\frac{F^{\prime }\left( u\right) }{G\left( u\right) }$ is monotonic and
\begin{eqnarray*}
\frac{F_{t}^{\prime }\left( u\right) }{G\left( u\right) }
&=&-\left( \frac{t^{2}}{4\pi ^{2}nu}\right) ^{\frac{1}{2}}\sqrt{n}u^{\sigma +1/4}+\sqrt{n}u^{\sigma +1/4} \\
&\geq &-\left( \frac{t^{2}}{4\pi ^{2}(1+\eta)yx}\right) ^{\frac{1}{2}}\sqrt{n}x^{\sigma +1/4}+\sqrt{n}x^{\sigma +1/4} \\
&\geq &\left( 1-\frac{1}{\sqrt{1+\eta }}\right) \sqrt{n}x^{\sigma +1/4} \\
&\gg &\sqrt{n}x^{\sigma +1/4}.
\end{eqnarray*}
Hence Lemma \ref{First drivitive test} gives
\begin{equation*}
\int_{x}^{N}u^{-\sigma -3/4}e\left( -\frac{t}{2\pi }\log u-\sqrt{nu}+1/8\right) \ll x^{-\sigma -1/4}n^{-\frac{1}{2}},
\end{equation*}
which yields
\begin{eqnarray*}
\mathfrak{M}_{2}\left( s\right) &\ll &x^{-\sigma+3/4}\sum_{\left( 1+\eta \right) y<n\leq \beta H^{2}}
\frac{d_{\alpha,\beta }\left(n;H\right)}{n^{5/4}} \\
%&\ll &x^{-\sigma +3/4}y^{-\frac{1}{4}}\log H \\
&\ll &t^{-\frac{1}{2}}x^{1-\sigma }\log H.
\end{eqnarray*}
\end{proof}

\begin{lem}\label{Flower M_3(s)}
For $\sigma >1/2,$ we have
\begin{equation*}
\mathfrak{M}_{3}\left( s\right) \ll_{\sigma } \left( x^{\frac{1}{2}-\sigma
}+x^{-\sigma }t^{\frac{1}{2}}\right) \log t.
\end{equation*}
\end{lem}

\begin{proof}
Write
\begin{eqnarray*}
\mathfrak{M}_{3}\left( s\right) &=&\frac{s}{2\pi \sqrt{2}}\int_{x}^{N}u^{-s-\frac{3}{4}}
\sum\limits_{n\leq y}\frac{d_{\alpha ,\beta }\left( n,H\right) }{n^{\frac{3}{4}}}e\left( -2\sqrt{nu}+1/8\right) du \\
&=&-\frac{1}{2\pi \sqrt{2}}\int_{x}^{N}\left( -s+1/4\right)
\sum_{n\leq y}\frac{d_{\alpha ,\beta }\left( n;H\right) }{n^{3/4}}e\left( -2\sqrt{nu}+1/8\right) u^{-s-3/4}du \\
&&+\frac{1}{8\pi \sqrt{2}}\int_{x}^{N}\sum_{n\leq y}\frac{d_{\alpha,\beta}\left( n;H\right) }{n^{3/4}}e\left( -2\sqrt{nu}+1/8\right) u^{-s-3/4}du \\
&=&-\frac{1}{2\pi \sqrt{2}}\sum_{n\leq y}\frac{d_{\alpha ,\beta}\left(n;H\right) }{n^{3/4}}
\int_{x}^{N}e\left( -2\sqrt{nu}+1/8\right) du^{-s+\frac{1}{4}} \\
&&+\frac{1}{8\pi \sqrt{2}}\sum_{n\leq y}\frac{d_{\alpha ,\beta}\left( n;H\right) }{n^{3/4}}\int_{x}^{N}e\left(-2\sqrt{nu}+1/8\right) u^{-s-3/4}du.
\end{eqnarray*}
\bigskip By partial integration, we have
\begin{equation}\label{Expression of flower M_3(s)}
\mathfrak{M}_{3}\left( s\right) =\mathfrak{M}_{31}\left( s\right)
+\mathfrak{M}_{32}\left( s\right) +\mathfrak{M}_{33}\left( s\right) +\mathfrak{M}_{34}\left( s\right)
\end{equation}
where

\begin{equation*}
\mathfrak{M}_{31}\left( s\right) =-\frac{N^{-s+1/4}}{2\pi \sqrt{2}}
\sum_{n\leq y}\frac{d_{\alpha ,\beta}\left( n;H\right) }{n^{3/4}}e\left( -2\sqrt{nN}+1/8\right) ,
\end{equation*}
\begin{equation*}
\mathfrak{M}_{32}\left( s\right)=\frac{x^{-s+1/4}}{2\pi \sqrt{2}}
\sum_{n\leq y}\frac{d_{\alpha ,\beta}\left( n;H\right) }{n^{3/4}}e\left(-2\sqrt{nx}+1/8\right),
\end{equation*}
\begin{eqnarray*}
&&\mathfrak{M}_{33}\left( s\right)=-\frac{i}{\sqrt{2}}\sum_{n\leq y}\frac{d_{\alpha ,\beta }\left(n;H\right) }{n^{1/4}}
\int_{x}^{N}u^{-\sigma -1/4}e\left( -\frac{t}{2\pi }\log u-2\sqrt{nu}-1/8\right) du
\end{eqnarray*}
and
\begin{eqnarray*}
&&\mathfrak{M}_{34}\left( s\right)=\frac{1}{8\pi \sqrt{2}}\sum_{n\leq y}\frac{d_{\alpha ,\beta }\left(n;H\right) }{n^{3/4}}
\int_{x}^{N}e\left( -2\sqrt{nu}-1/8\right) u^{-s-3/4}du.
\end{eqnarray*}
Using $d_{\alpha ,\beta }\left( n;H\right) \leq d\left( n\right) $
and trivial estimates, it is easy to get
\begin{equation}\label{Upper bound of flower M_31,M_32}
\mathfrak{M}_{31}\left( s\right) ,\mathfrak{M}_{32}\left( s\right)\ll _{\sigma }x^{\frac{1}{2}-\sigma }\log t
\end{equation}
and
\begin{equation}\label{Upper bound of flower M_34}
\mathfrak{M}_{34}\left( s\right) \ll _{\sigma }x^{-\sigma +\frac{1}{4}}y^{\frac{1}{4}}\log y\ll x^{-\sigma }t^{\frac{1}{2}}\log t.
\end{equation}
Now we deal with $\mathfrak{M}_{33}\left( s\right) $. In Lemma
\ref{First drivitive test}, let
\begin{equation*}
H\left( u\right) =1,G\left( u\right) =u^{-\sigma -1/4}\ \ \ \ {\rm and}\ \ \ \
F_{t}\left( u\right) =-\frac{t}{2\pi }\log u-2\sqrt{nu}-1/8,
\end{equation*}
then we have
\begin{equation*}
F_{t}^{\prime }\left( u\right) =-\frac{t}{2\pi u}-\sqrt{\frac{n}{u}}
\end{equation*}
and
\begin{equation*}
\frac{F_{t}^{\prime }\left( u\right) }{G\left( u\right) }=-\frac{t}{2\pi }u^{\sigma -\frac{3}{4}}-\sqrt{n}u^{\sigma -\frac{1}{4}}.
\end{equation*}
Obviously,
\begin{equation*}
\frac{F_{t}^{\prime }\left( u\right) }{G\left( u\right)
}<-\sqrt{n}u^{\sigma -\frac{1}{4}}\leq -\sqrt{n}x^{\sigma
-\frac{1}{4}}.
\end{equation*}
Noting
\begin{equation*}
\left( \frac{F_{t}^{\prime }\left( u\right) }{G\left( u\right)
}\right)
^{\prime }=-\left( \sigma -\frac{3}{4}\right) \frac{t}{2\pi }u^{\sigma -%
\frac{7}{4}}-\left( \sigma -\frac{1}{4}\right) \sqrt{n}u^{\sigma -\frac{5}{4}%
},
\end{equation*}%
let $u_{0}=\frac{\left( 3/4-\sigma \right) t}{\left( \sigma
-1/4\right) 2\pi
\sqrt{n}}$ be the root of $\left( \frac{F_{t}^{\prime }\left( u\right) }{G\left( u\right) }\right) ^{\prime }=0.$
If $u_{0}\in \lbrack x,N],$ then $\frac{F_{t}^{\prime }\left( u\right) }{G\left( u\right) }$ is monotonic in
$[x,u_{0}]$ and $[u_{0},N]$ respectively, otherwise$\frac{F_{t}^{\prime }\left( u\right) }{G\left( u\right) }$ is
monotonic in $[x,N].$ In either case, Lemma \ref{First drivitive test} is valid and gives
\begin{equation*}
\int_{x}^{N}u^{-\sigma -1/4}e\left( -\frac{t}{2\pi }\log u-2\sqrt{nu}-1/8\right) \ll n^{-\frac{1}{2}}x^{\frac{1}{4}-\sigma },
\end{equation*}
which yields
\begin{equation}\label{Upper bound of flower M_33}
\mathfrak{M}_{33}\left( s\right) \ll \sum_{n\leq y}\frac{d\left( n\right) }{n^{1/4}}n^{-\frac{1}{2}}x^{\frac{1}{4}-\sigma }
\ll x^{\frac{1}{4}-\sigma }y^{1/4}\log y\ll x^{-\sigma}t^{\frac{1}{2}}\log t.
\end{equation}
Then Lemma \ref{Flower M_3(s)} follows from collecting
(\ref{Expression of flower M_3(s)})-(\ref{Upper bound of flower
M_33}).
\end{proof}

\begin{lem}\label{Flower M_4(s)}
For $\sigma >1/2,$ we have%
\begin{equation*}
\mathfrak{M}_{4}\left( s\right) \ll x^{1/2-\sigma }\log H.
\end{equation*}
\end{lem}
\begin{proof}
Write
\begin{equation*}
\mathfrak{M}_{4}\left( s\right) =\frac{s}{2\pi \sqrt{2}}\sum_{y<n\leq \beta H^{2}}
\frac{d_{\alpha ,\beta }\left( n;H\right) }{n^{3/4}}\int_{x}^{N}u^{-\sigma -3/4}e\left( -\frac{t}{2\pi }\log u-2\sqrt{nu}+1/8\right) du
\end{equation*}
In Lemma \ref{First drivitive test}, taking
\begin{equation*}
G\left( u\right) =u^{-\sigma -3/4}
\end{equation*}
and
\begin{equation*}
F_{t}\left( u\right) =-\frac{t}{2\pi }\log u-2\sqrt{nu}+1/8,
\end{equation*}
we have
\begin{equation*}
F_{t}^{\prime }\left( u\right) =-\frac{t}{2\pi u}-\sqrt{\frac{n}{u}}
\end{equation*}
and%
\begin{equation*}
\frac{F_{t}^{\prime }\left( u\right) }{G\left( u\right) }=-\frac{t}{2\pi }
u^{\sigma -1/4}-\sqrt{n}u^{\sigma +1/4}.
\end{equation*}
Thus $\frac{F^{\prime }\left( u\right) }{G\left( u\right) }$ is
monotonic and
\begin{eqnarray*}
\frac{F_{t}^{\prime }\left( u\right) }{G\left( u\right) } &=&-\frac{t}{2\pi }u^{\sigma -1/4}-\sqrt{n}u^{\sigma +1/4} \\
&<&-\sqrt{n}x^{\sigma +1/4}
\end{eqnarray*}
Hence Lemma \ref{First drivitive test} gives
\begin{equation*}
\int_{x}^{N}u^{-\sigma -3/4}e\left( -\frac{t}{2\pi }\log u-\sqrt{nu}
+1/8\right) \ll x^{-\sigma -1/4}n^{-\frac{1}{2}},
\end{equation*}
which yields
\begin{equation*}
\mathfrak{M}_{4}\left( s\right) \ll x^{-\sigma +3/4}\sum_{y<n\leq \beta H^{2}}
\frac{d_{\alpha ,\beta }\left( n;H\right) }{n^{5/4}}\ll x^{1/2-\sigma }\log H.
\end{equation*}
\end{proof}

\begin{lem}\label{Lemma Flower M_1}
For $\sigma >1/2,$ we
have
\begin{equation*}
\mathfrak{M}_{1}\left( s\right) =\chi ^{2}\left( s\right)\sum_{n\leq y}d_{\alpha ,\beta }\left( n\right) n^{s-1}
+O\left(x^{1/2-\sigma }\log t\right) .
\end{equation*}
\end{lem}

\begin{proof}
\bigskip Similar to the the proof of Lemma \ref{Flower M_3(s)}, we rewrite
$\mathfrak{M}_{1}\left( s\right) $ as
\begin{eqnarray*}
\mathfrak{M}_{1}\left( s\right) &=&\frac{s}{2\pi \sqrt{2}}\int_{x}^{N}u^{-s-\frac{3}{4}}
\sum\limits_{n\leq \left( 1+\eta \right)y}\frac{d_{\alpha ,\beta }\left( n,H\right)}{n^{\frac{3}{4}}}e\left( 2\sqrt{nu}-1/8\right) du\\
&=&-\frac{1}{2\pi \sqrt{2}}\int_{x}^{N}\left( -s+1/4\right)\sum_{n\leq\left( 1+\eta \right) y}
\frac{d_{\alpha ,\beta }\left( n;H\right) }{n^{3/4}}e\left( 2\sqrt{nu}-1/8\right) u^{-s-3/4}du \\
&&+\frac{1}{8\pi \sqrt{2}}\int_{x}^{N}\sum_{n\leq \left( 1+\eta \right) y}
\frac{d_{\alpha ,\beta }\left( n;H\right) }{n^{3/4}}e\left( 2\sqrt{nu}-1/8\right) u^{-s-3/4}du \\
&=&-\frac{1}{2\pi \sqrt{2}}\sum_{n\leq \left( 1+\eta \right) y}\frac{d_{\alpha ,\beta }\left( n;H\right) }{n^{3/4}}
\int_{x}^{N}e\left( 2\sqrt{nu}-1/8\right) du^{-s+\frac{1}{4}} \\
&&+\frac{1}{8\pi \sqrt{2}}\sum_{n\leq \left( 1+\eta \right) y}\frac{d_{\alpha ,\beta }\left( n;H\right) }{n^{3/4}}
\int_{x}^{N}e\left( 2\sqrt{nu}-1/8\right) u^{-s-3/4}du.
\end{eqnarray*}
By partial integration, we have
\begin{equation}\label{Flower M_1(s)=}
\mathfrak{M}_{1}\left( s\right) =\mathfrak{M}_{11}\left( s\right) +\mathfrak{M}_{12}\left( s\right)
+\mathfrak{M}_{13}\left( s\right) +\mathfrak{M}_{14}\left( s\right) ,
\end{equation}
where
\begin{equation*}
\mathfrak{M}_{11}\left( s\right)=-\frac{1}{i\sqrt{2}}\sum_{n\leq\left( 1+\eta \right) y}d_{\alpha ,\beta }
\left( n;H\right)n^{-\frac{1}{4}}I_{n}
\end{equation*}
with
\begin{equation*}
I_{n}=\int_{x}^{N}u^{-\sigma -\frac{1}{4}}e\left( -\frac{t}{2\pi }\log u
+2\sqrt{nu}-\frac{1}{8}\right) du,
\end{equation*}
\begin{equation*}
\mathfrak{M}_{12}\left( s\right)=-\frac{1}{2\pi\sqrt{2}}\sum_{n\leq\left( 1+\eta \right) y}
\frac{d_{\alpha ,\beta }\left( n;H\right) }{n^{3/4}}e\left( 2\sqrt{nN}-1/8\right) N^{-s+\frac{1}{4}},
\end{equation*}
\begin{equation*}
\mathfrak{M}_{13}\left( s\right)=\frac{1}{2\pi\sqrt{2}}\sum_{n\leq \left( 1+\eta \right) y}\frac{d_{\alpha ,\beta}
\left( n;H\right) }{n^{3/4}}e\left( 2\sqrt{nx}-1/8\right) x^{-s+\frac{1}{4}}
\end{equation*}
and
\begin{equation*}
\mathfrak{M}_{14}\left( s\right)=\frac{1}{8\pi\sqrt{2}}\sum_{n\leq \left(1+\eta \right) y}
\frac{d_{\alpha ,\beta }\left( n;H\right) }{n^{3/4}}\int_{x}^{N}e\left( 2\sqrt{nu}-1/8\right) u^{-s-3/4}du.
\end{equation*}
Note that $\eta >0$ is a fixed, sufficiently small constant, then by $d_{\alpha ,\beta }\left( n;H\right) \leq d\left( n\right) $
and trivial estimates, we get
\begin{equation}\label{Flower M_1,234}
\mathfrak{M}_{12}\left( s\right) ,\mathfrak{M}_{13}\left( s\right) ,
\mathfrak{M}_{14}\left( s\right) \ll _{\sigma}x^{\frac{1}{4}-\sigma }y^{\frac{1}{4}}
\log y\ll _{\sigma}x^{-\sigma }t^{\frac{1}{2}}\log t.
\end{equation}
Now only $\mathfrak{M}_{11}\left( s\right) $ is left. In Chapter 4 of \cite{Ivic} (Page 108-110), Ivic discussed $I_{n}$ and showed
\begin{equation*}
-\frac{1}{i\sqrt{2}}\sum_{n\leq \left( 1+\eta \right) y}d\left( n\right) n^{-\frac{1}{4}}I_{n}
=\chi ^{2}\left( s\right) \sum_{n\leq y}d\left(n\right) n^{s-1}+O\left( x^{\frac{1}{2}-\sigma }\log t\right),
\end{equation*}
where $\chi\left( s\right)$ is given by \rm{(\ref{Functional equation of Riemann zeta-function})}. Replacing $d\left( n\right) $ by
$d_{\alpha ,\beta }\left(n;H\right) ,$ the same argument is also valid, which gives
\begin{eqnarray*}
\mathfrak{M}_{11}\left( s\right) &=&-\frac{1}{i\sqrt{2}}\sum_{n\leq\left(1+\eta \right) y}d_{\alpha ,\beta }
\left( n;H\right) n^{-\frac{1}{4}}I_{n} \\
&=&\chi ^{2}\left( s\right) \sum_{n\leq y}d_{\alpha ,\beta }\left(n;H\right) n^{s-1}+O\left( x^{1/2-\sigma }\log t\right).
\end{eqnarray*}
Take $H=T^{B}$ with $B>3A>0$ being a constant, sufficiently large, then we have
\begin{eqnarray}\label{Flower M_11}
\mathfrak{M}_{11}\left( s\right) &=&\chi ^{2}\left( s\right)
\sum_{n\leq y}\left( \underset{n=hk}{\sum_{1\leq h\leq H}\sideset{}{^{\prime \prime }}
\sum\limits_{\ h\alpha \leq k\leq h\beta }}1\right)n^{s-1}+O\left( x^{1/2-\sigma }\log t\right)\\
&=&\chi ^{2}\left( s\right) \sum_{n\leq y}\left(\underset{n=hk}{\sum_{1\leq h\leq H}
\sum\limits_{h\alpha \leq k\leq h\beta }}1\right) n^{s-1}+O\left(\left\vert \chi \left( s\right) \right\vert ^{2}
\sum_{h\ll \sqrt{y}}h^{2\sigma -2}\right)\notag\\
\indent &&+O\left( x^{1/2-\sigma }\log T\right)  \notag \\
&=&\chi ^{2}\left( s\right) \sum_{n\leq y}d_{\alpha ,\beta }\left(n\right) n^{s-1}
+O\left( t^{1-2\sigma }y^{\sigma -1/2}+x^{1/2-\sigma}\log T\right)\notag \\
&=&\chi ^{2}\left( s\right) \sum_{n\leq y}d_{\alpha ,\beta }\left(n\right) n^{s-1}+O\left( x^{1/2-\sigma }\log T\right),  \notag
\end{eqnarray}
where we used

\begin{equation}%\label{estimate of ka(s)}
\chi \left( \sigma +it\right) =\left( \frac{2\pi }{t}\right) ^{\sigma +it-\frac{1}{2}}e^{i\left( t+\frac{\pi }{4}\right) }
\left( 1+O\left(t^{-1}\right) \right),\ \ \ \ \mathrm{ for }\ t\geq 2.
\end{equation}
Combining (\ref{Flower M_1(s)=})-(\ref{Flower M_11}) gives Lemma
\ref{Lemma Flower M_1}.
\end{proof}

\section{Out line for the proof of Theorem \ref{Theorem 1.2}}

A primitive Pythagorean triangle is a triple $(a,b,c)$ of natural numbers with $a^{2}+b^{2}=c^{2},a<b$ and $\gcd (a,b,c)=1.$ Let $P(x)$ denote the number of primitive Pythagorean triangles with perimeter less than $x.$
D. H. Lehmer \cite{H.Lehmer} showed
\begin{equation*}
P\left( x\right) =\frac{\log 2}{\pi ^{2}}x+O\left( x^{1/2}\log x\right)
\end{equation*}
which was revisited by J. Lambek and L. Moser in \cite{Lambek}. The exponents
$1/2$ in the error term can not be reduced because the current
technique depends on the best zero-free regions of the Riemann zeta
function, which hard to be improved. In \cite{Liu}, the author showed if
Riemann Hypothesis (RH) is true, then (\ref{old theorem on triangles}) holds.
Let
\begin{equation*}
r\left( n\right) =\sum_{\substack{ 2d^{2}+2dl=n \\ l<d}}1=\sum_{\substack{2dl=n \\ d<l<2d}}1
\end{equation*}
and
\begin{equation*}
Z\left( s\right) =\sum_{n=1}^{\infty }\frac{r\left( n\right)}{n^{s}},\text{ for }\Re s>1.
\end{equation*}
We can prove that $Z(s)$ has an analytic continuation to $\sigma>1/3$ and has two simple poles at $s=1,\frac{1}{2}.$ The exponent
$\frac{5805}{15408}$ in (\ref{old theorem on triangles}) depends on the estimate of the following type exponential sum
\begin{equation}\label{Exponential sum}
\sum_{m\sim M}\mu \left( m\right) \sum_{n\sim N}a_{n}e\left( \frac{cx^{\frac{1}{2}}n^{\frac{1}{2}}}{m}\right)
\end{equation}
with $a_{n}\ll 1$ and $c$ being a constant. Here the ranges of $M,\ N$
are determined by the smallest $\sigma $ such that
\begin{equation}\label{Z(s)}
\int_{T}^{2T}\left\vert Z\left( \sigma +it\right) \right\vert dt\ll_{\sigma ,\varepsilon }T^{1+\varepsilon }
\end{equation}
holds for any $\varepsilon >0.$ In \cite{Liu}, the author showed
$\sigma >\frac{1064}{1644}=0.6472\cdots $ is admissible. Then by
estimating the exponential sum (\ref{Exponential sum}) for $M\leq
x^{\frac{651}{1926}},$ $N\leq x^{\frac{3798}{15408}}$, we get
(\ref{old theorem on triangles}). In the review of \cite{Liu}, R. C. Baker mentioned that using
the method in his paper \cite{Baker 2}, it is possible to
prove $\sigma >\frac{3}{5}=0.6,$ which implies an improvement of (\ref{old theorem on triangles}).
Now by Theorem \ref{Main theorem}, we have (\ref{Z(s)}) holds for
any $\sigma>\frac{1}{2}$, which forces us to deal with exponential sum (\ref{Exponential sum}) for $M,$ $N\leq $ $x^{\frac{1}{4}+\varepsilon }.$
However, the estimate in this range has been investigated carefully by R. C. Baker in \cite{Baker 1},
which yields Theorem \ref{Theorem 1.2}.

\bibliographystyle{plainnat}

\bigskip

\bigskip

\end{document}